\documentclass[11pt]{amsart}
\usepackage{graphics}
\usepackage{amsfonts,amsmath}
\begin{document}

 \newtheorem{thm}{Theorem}[section]
 \newtheorem{cor}[thm]{Corollary}
 \newtheorem{lem}[thm]{Lemma}{\rm}
 \newtheorem{prop}[thm]{Proposition}

 \newtheorem{defn}[thm]{Definition}{\rm}
 \newtheorem{assumption}[thm]{Assumption}
 \newtheorem{rem}[thm]{Remark}
 \newtheorem{ex}{Example}\numberwithin{equation}{section}

\def\supp{{\rm supp}}
\def\la{\langle}
\def\ra{\rangle}
\def\e{{\rm e}}
\def\x{\mathbf{x}}
\def\by{\mathbf{y}}
\def\bz{\mathbf{z}}
\def\F{\mathcal{F}}
\def\R{\mathbb{R}}
\def\T{\mathbf{T}}
\def\N{\mathbb{N}}
\def\K{\mathbf{K}}
\def\bK{\overline{\mathbf{K}}}
\def\Q{\mathbf{Q}}
\def\M{\mathbf{M}}
\def\O{\mathbf{O}}
\def\C{\mathbf{C}}
\def\P{\mathbf{P}}
\def\Z{\mathbf{Z}}
\def\H{\mathcal{H}}
\def\A{\mathbf{A}}
\def\V{\mathbf{V}}
\def\AA{\overline{\mathbf{A}}}
\def\B{\mathbf{B}}
\def\c{\mathbf{c}}
\def\L{\mathbf{L}}
\def\bS{\mathbf{S}}
\def\H{\mathcal{H}}
\def\I{\mathbf{I}}
\def\Y{\mathbf{Y}}
\def\X{\mathbf{X}}
\def\G{\mathbf{G}}
\def\f{\mathbf{f}}
\def\z{\mathbf{z}}
\def\y{\mathbf{y}}
\def\d{\hat{d}}
\def\bx{\mathbf{x}}
\def\y{\mathbf{y}}
\def\v{\mathbf{v}}
\def\g{\mathbf{g}}
\def\w{\mathbf{w}}
\def\b{\mathcal{B}}
\def\a{\mathbf{a}}
\def\b{\mathbf{b}}
\def\h{\mathbf{h}}
\def\q{\mathbf{q}}
\def\u{\mathbf{u}}
\def\s{\mathcal{S}}
\def\cc{\mathcal{C}}
\def\co{{\rm co}\,}
\def\cop{{\rm COP}\,}
\def\tg{\tilde{f}}
\def\tx{\tilde{\x}}
\def\supmu{{\rm supp}\,\mu}
\def\supnu{{\rm supp}\,\nu}
\def\m{\mathcal{M}}
\def\s{\mathcal{S}}
\def\k{\mathcal{K}}
\def\la{\langle}
\def\ra{\rangle}
\def\dd{\dagg\dagg}
\def\l{{\rm !}}
\def\psd{{\rm Psd}}

\title{The $\K$-moment problem for continuous linear functionals}

\author{Jean B. Lasserre}
\address{LAAS-CNRS and Institute of Mathematics\\
University of Toulouse\\
LAAS, 7 avenue du Colonel Roche\\
31077 Toulouse C\'edex 4,France}
\email{lasserre@laas.fr}
\date{}

\begin{abstract}
Given a closed (and non necessarily compact) basic semi-algebraic set
$\K\subseteq\R^n$, we solve the $\K$-moment problem for continuous linear functionals.
Namely, we introduce a weighted $\ell_1$-norm $\ell_\w$ on $\R[\x]$, and
show that the $\ell_\w$-closures  of the preordering $P$ and quadratic module $Q$
(associated with the generators of $\K$)
is the cone $\psd(\K)$ of polynomials nonnegative on $\K$. We also prove that
$P$ and $Q$ solve the $\K$-moment problem for $\ell_\w$-continuous linear functionals
and completely characterize those $\ell_\w$-continuous linear functionals nonnegative on 
$P$ and $Q$ (hence on $\psd(\K)$). When $\K$ has a nonempty interior
we also provide in explicit form a canonical
$\ell_\w$-projection $g^\w_f$ for any polynomial $f$,
on the (degree-truncated) preordering or quadratic module.
Remarkably, the support of $g^\w_f$ is very sparse and does not depend on $\K$! 
This enables us to provide an explicit Positivstellensatz on $\K$. 
At last but not least, we provide a simple characterization of polynomials nonnegative on
$\K$, which is crucial in proving the above results.
\end{abstract}



\maketitle

\section{Introduction}

This paper is concerned with basic closed semi-algebraic sets $\K\subseteq\R^n$
and the {\it preordering} $P(g)$ and {\it quadratic module} $Q(g)$
associated with the finite family of polynomials $(g_j)$, $j\in J$, that generate $\K$.
In particular, when $\K=\R^n$ then the latter two convex cones coincide with the cone $\Sigma[\x]$ of
sums of squares (s.o.s.) of polynomials.
The convex cones $P(g)$ and $Q(g)$ (which are subcones of the convex cone $\psd(\K)$ of polynomials nonnegative on $\K$) are of practical importance because on the one hand nonnegative polynomials are ubiquitous but on the other hand,
polynomials in $P(g)$ or $Q(g)$ are much easier to handle.
For instance, and in contrast with nonnegative polynomials, checking whether
a given polynomial is in $P_d(g)\,(\subset P(g))$ or $Q_d(g)\,(\subset Q(g))$ (i.e., with an {\it a priori}  degree bound 
$d$ on its representation) can be done efficiently by solving a semidefinite program, a powerful technique of convex optimization.

The celebrated $\K$-moment problem is concerned with characterizing all real sequences
$\y=(y_\alpha)$, $\alpha\in\N^n$, which can be realized as the moment sequence of some finite Borel measure
on $\K$. For such sequences $\y$, the linear form $L_\y$ associated with $\y$ is called a {\it $\K$-moment functional}
and Haviland's theorem states that $L_\y$ is a $\K$-moment functional if and only if
$L_\y(f)\geq0$ for all $f\in \psd(\K)$. And so one says that $P(g)$ (resp. $Q(g)$) solves the $\K$-moment problem if the $\K$-moment functionals are those which satisfy $L_\y(f)\geq 0$ on $P(g)$ (resp. $Q(g)$). 
However, this is true if and only if $P(g)$ (resp. $Q(g)$) is dense in $\psd(\K)$ for the finest locally convex topology on $\R[\x]$.
When $\K$ is compact, the $\K$-moment problem was  completely solved (with $\overline{P(g)}=\psd(\K)$)
in Schm\"udgen \cite{schmudgen} and soon refined to $\overline{Q(g)}=\psd(\K)$ in Putinar \cite{putinar}
for Archimedean quadratic modules $Q(g)$. 

Since then, recent contributions have tried to better understand (in even a more general framework) the links between
$\psd(\K)$ and closures (and sequential closures) of preorderings and quadratic modules,
one important goal being to extend (or provide analogues of) Schm\"udgen and Putinar's Positivstellensatz\"e \cite{schmudgen,putinar} to cases where $\K$ is not compact. In particular, Scheiderer \cite{claus} has shown
rather negative results in this direction.
%A particularly important issue is to understand the link between the closure and the sequential closures of $P(g)$ and $Q(g)$. 
For more details on those recent results, the interested is referred to e.g. Powers and Scheiderer \cite{powers},
Kuhlmann and Marshall \cite{salma0}, Kuhlmann et al. \cite{salma}, and Cimpric et al. \cite{cimpric}.

On the one hand, {\it all} linear functionals are continuous in
the finest locally convex topology, on the other hand, the negative results of Scheiderer \cite{claus}
suggest that solving the $\K$-moment problem via preorderings or quadratic modules is possible only in specific cases,
and so this topology is not the most appropriate in general. So why not rather consider other topologies on $\R[\x]$ and
the $\K$-moment problem for linear functionals $L_\y$ that are {\it continuous} for these topologies?
This is the point of view taken in Ghasemi et al. \cite{lastsalma} where the authors
consider certain (weighted) norm-topologies and show that the closure of the cone of
sums of squares is $\psd([-1,1]^n)$ as did Berg \cite{berg2} for the $\ell_1$-norm.
Notice that this was also the point of view taken in Schm\"udgen \cite{schmudgen3}
for a class of non commutative (enveloping) $*$-algebra; in the latter context \cite{schmudgen3}, the author
proves that  the cone of {\it strongly positive} elements
is the closure of the (smallest) cone of sums of squares in certain topologies.

{\bf Contribution.} In view of the negative results in Scheiderer \cite{claus}, 
%when using the finest locally convex topology,
we also consider the above mentioned viewpoint and look at the $\K$-moment problem by using 
a particular weighted $\ell_1$-norm on $\R[\x]$ (denoted $\ell_\w$ for a certain sequence $\w:\N^n\to\R_{>0}$) rather
than the usual finest locally convex topology.
In this framework we solve the $\K$-moment problem for basic closed semi-algebraic sets $\K\subseteq\R^n$
in the following sense. We prove that
(a) the $\ell_\w$-closure of $P(g)$ and $Q(g)$ is exactly $\psd(\K)$, and (b) 
$P(g)$ and $Q(g)$ solve the moment problem, i.e.,
the $\K$-moment ($\ell_\w$-continuous) functionals are those
$L_\y$ that are $\ell_\w$-continuous and nonnegative on $Q(g)$ (or $P(g)$).
In fact, such linear functionals $L_\y$ are characterized by:

- $L_\y(h^2 g_j)\geq0$ for all $h\in\R[\x]$, and every generator $g_j$ of $\K$.

- $\exists M>0$ such that $\vert y_\alpha\vert\leq \,M\,w_\alpha$ for all $\alpha\in\N^n$.

$\bullet$ Next, when $\K$ has a nonempty interior, there exist $\ell_\w$-projections 
of a polynomial onto $P_d(g)$ and $Q_d(g)$, where $P_d(g)$ (resp. $Q_d(g)$) denotes the subcone of elements of $P(g)$ (resp. of $Q(g)$) which
have a degree bound $d$ in their representation. In general these projections are not unique
but we provide a {\it canonical} $\ell_\w$-projection $g_f$ for any polynomial $f$,
which takes a remarkably simple and ``sparse" form, and particularly when $\ell_\w$ is the usual $\ell_1$-norm, in which case
\begin{equation}
\label{simpleform}
g_f\,=\,f+\lambda_0+\sum_{i=1}^n\lambda_i\, x_i^{2d},\end{equation}
for some nonnegative vector $\lambda\in\R^{n+1}$. In other words,
the {\it support} $\Vert g_f\Vert_0$ of $g_f$ does not depend on $\K$ and does not depend on $d$ either!
The dependence of $g_f$ on the 
$g_j$'s that define $\K$ is only through the 
coefficients $(\lambda^*_j)$. In addition, the support is very sparse since $\Vert g_f\Vert_0\leq \Vert f\Vert_0+n+1$.
This confirms the property of the $\ell_1$-norm 
with respect to sparsity, already observed in other contexts. Minimizing the $\ell_1$-norm
aims at finding a solution with small support (where $\Vert\x\Vert_0=\#\{i:x_i\neq0\}$).
Finally, the vector $\lambda$ in (\ref{simpleform}) is an optimal solution of an explicit semidefinite program, and so can be computed efficiently.

$\bullet$ We also provide a canonical $\ell_1$-projection of $f$
onto $\overline{P(g)\cap\R[\x]_{2d}}$ which is again of the form (\ref{simpleform}), and use this result to characterize
the sequential closure $P(g)^{\ddag}$ of $P(g)$ for the finest locally convex topology. Namely,
\[P(g)^{\ddag}\,=\,\left\{f\in\R[\x]\::\:\exists\,d\mbox{ s.t. }\forall\epsilon>0,\:f+\epsilon\left(1+\sum_{i=1}^nx_i^{2d}\right)\in P(g)\,\right\},\]
and the same statement is true for the quadratic module $Q(g)$. This latter result exhibits the particularly simple form 
$q:=(1+\sum_{i=1}^nx_i^{2d})$ for the possible polynomials $q$ in the characterization
of $P(g)^{\ddag}$ (and $Q(g)^{\ddag}$) provided in e.g. \cite{cimpric,salma}; e.g., 
in \cite{salma} it is stated that one may take the polynomial $(1+\Vert\x\Vert^2)^s$ for some $s$.

$\bullet$ Thanks to the characterization of canonical $\ell_\w$-projections, we finally obtain
a Positivstellensatz on $\K$ of the following form:
$f\in\psd(\K)$ if and only if for every $\epsilon>0$, there is some $d\in\N$
such that the polynomial $f+\epsilon(1+\sum_{i=1}^n\sum_{k=1}^d x^{2k}_i/(2k)\l)$ is in $P(g)$ (or $Q(g)$).

$\bullet$ At last but not least, and crucial for the above results,
we prove a result of independent interest, 
concerned with sequences $\y=(y_\alpha)$, $\alpha\in\N^n$, that have a finite representing
Borel measure $\mu$. Namely, we prove that a polynomial $f$ is nonnegative on the support of $\mu$
if and only if $\int h^2fd\mu\geq0$ for all $h\in\R[\x]$.

The paper is organized as follows. In Section 2,  and after introducing the notation and definitions,
we present the intermediate result mentioned above. In Section 3 we show that $P(g)$ and $Q(g)$ solve the $\K$-moment problem for $\ell_\w$-continuous linear functionals. In section 4, we provide explicit expressions for the 
canonical $\ell_\w$-projections onto $P(g)$, $Q(g)$ and their truncated versions. Moreover,
we characterize the sequential closures $Q(g)^{\ddag}$ and $P(g)^{\ddag}$ and
we end up with a Positivstellensatz for
$\K$.

\section{Notation, definitions and preliminaries}

\subsection{Notation and definitions}
The notation $\mathcal{B}$ stands for the Borel $\sigma$-field of $\R^n$.
Let $\R[\x]$ (resp. $\R[\x]_d$) denote the ring of real polynomials in the variables
$\x=(x_1,\ldots,x_n)$ (resp. polynomials of degree at most $d$), whereas $\Sigma[\x]$ (resp. $\Sigma[\x]_d$) denotes 
its subset of sums of squares (s.o.s.) polynomials (resp. of s.o.s. of degree at most $2d$).
For every
$\alpha\in\N^n$ the notation $\x^\alpha$ stands for the monomial $x_1^{\alpha_1}\cdots x_n^{\alpha_n}$,
$\vert\alpha\vert$ stands for the integer $a:=\alpha_1+\cdots+\alpha_n$, and
$\alpha\l$ stands for the integer $a\l$ 
For an arbitrary set $\bS\subseteq\R^n$, let $\psd(\bS)$ denote the convex cone of polynomials
that are nonnegative on $\bS$.

For every $i\in\N$, let $\N^{n}_d:=\{\beta\in\N^n:\sum_j\beta_j\leq d\}$ whose cardinal is $s(d)={n+d\choose n}$.
A polynomial $f\in\R[\x]$ is written 
\[\x\mapsto f(\x)\,=\,\sum_{\alpha\in\N^n}\,f_\alpha\,\x^\alpha,\]
and $f$ can be identified with its vector of coefficients $\f=(f_\alpha)$ in the canonical basis 
$(\x^\alpha)$, $\alpha\in\N^n$.  The {\it support} of $f\in\R[\x]$ is the set $\{\alpha\in\N^n:f_\alpha\neq0\}$ and let
$\Vert f\Vert_0:= {\rm card}\,\{\alpha:f_\alpha\neq0\}$. 
Denote by $\Vert f\Vert_1$ the $\ell_1$-norm $\sum_{\alpha}\vert f_\alpha\vert$
of the coefficient vector $\f$. 

Crucial in the sequel is the use of the following $\ell_\w$-norm which is a {\it weighted} $\ell_1$-norm defined from
the sequence $\w=(w_\alpha)$, $\alpha\in\N^n$, where $w_\alpha:=(2\lceil \vert\alpha\vert/2\rceil)\l$.
Namely, denote by $\Vert f\Vert_\w$ the $\ell_\w$-norm $\sum_{\alpha}w_\alpha\vert f_\alpha\vert$
of the coefficient vector $\f$; hence
the $\ell_1$-norm corresponds to the case where $w_\alpha=1$ for all $\alpha\in\N^n$.
Both $\ell_1$ and $\ell_\w$ also define a norm on $\R[\x]_d$.

Let $\s^p\subset\R^{p\times p}$ denote the space of real $p\times p$ symmetric matrices.
 For any two matrices $\A,\B\in\s^p$, the notation $\A\succeq0$ (resp. $\succ0$) stands for $\A$ is positive semidefinite
 (resp. positive definite), and the notation
$\la \A,\B\ra$ stands for ${\rm trace}\,\A\B$.

Let $\v_d(\x)=(\x^\alpha)$, $\alpha\in\N^n_d$, and let
$\B^0_\alpha\in\R^{s(d)\times s(d)}$ be real symmetric matrices such that
\begin{equation}
\label{balpha}
\v_d(\x)\,\v_d(\x)^T\,=\,\sum_{\alpha\in\N^n_{2d}}\x^\alpha\,\B^0_\alpha.\end{equation}
Recall that a polynomial $g\in\R[\x]_{2d}$ is a s.o.s. if and only if there exists a real positive semidefinite matrix
$\X\in\R^{s(d)\times s(d)}$ such that
\[g_\alpha\,=\,\langle \X,\B^0_\alpha\ra,\qquad \forall \alpha\in\N^n_{2d}.\]

Let $g_j\in\R[\x]$, $j=0,1,\ldots,m$, with $g_0$ being the constant polynomial $g_0(\x)=1$ for all $\x$,
and let $\K\subseteq\R^n$ be the basic closed semi algebraic set:
\begin{equation}
\label{setk}
\K:=\{\x\in\R^n\,:\, g_j(\x)\geq0,\: j=1,\ldots,m\},
\end{equation}
For every $J\subseteq\{1,\ldots,m\}$ let $g_J:=\prod_{k\in J}g_k$, with the convention $g_\emptyset:=1$,
and let $v_J:=\lceil ({\rm deg}\,g_J)/2\rceil$ (with $v_j:=v_{\{j\}}$).
\begin{defn}
With $\K$ as in (\ref{setk}),  let $P(g),Q(g)\subset\R[\x]$ and $P_k(g),Q_k(g)\subset\R[\x]_{2k}$ be the convex cones:
\begin{eqnarray*}
P(g)&:=&\left\{\sum_{J\subseteq\{1,\ldots,m\}}\sigma_J\,g_J\::\: \sigma_J\in\Sigma[\x],\: J\subseteq\{1,\ldots,m\}\right\},\\
P_k(g)&:=&\left\{\sum_{J\subseteq\{1,\ldots,m\}}\sigma_J\,g_J\::\: \sigma_J\in\Sigma[\x]_{k-v_J},\: J\subseteq\{1,\ldots,m\}\right\},\\
Q(g)&:=&\left\{\sum_{j=0}^m \sigma_j\,g_j\::\: \sigma_j\in\Sigma[\x],\: j=1,\ldots,m\right\},\\
Q_k(g)&:=&\left\{\sum_{j=0}^m \sigma_j\,g_j\::\: \sigma_j\in\Sigma[\x]_{k-v_j},\: j=1,\ldots,m\right\}.
\end{eqnarray*}
\end{defn}
The set $P(g)$ (resp. $Q(g)$) is a convex cone called the {\it preordering} (resp. the {\it quadratic module}) associated with the $g_j$'s. 
Obviously, if $h\in P(g)$ (resp. $h\in Q(g)$), the associated s.o.s. weights $\sigma_J$'s (resp. $\sigma_j$'s) of its representation provide a certificate of nonnegativity of $h$ on $\K$. The convex cone $P_k(g)$ (resp. $Q_k(g)$) is the
subset of elements $h\in P(g)$ (resp. $h\in Q(g)$) with a {\it degree bound $2k$} certificate. Observe that 
$P_k(g)\subset P(g)\cap\R[\x]_{2k}$ and $Q_k(g)\subset Q(g)\cap\R[\x]_{2k}$.

\subsection*{Moment matrix} With a sequence $\y=(y_\alpha)$, $\alpha\in\N^n$,
let $L_\y:\R[\x]\to\R$ be the linear functional
\[f\quad (=\sum_{\alpha}f_{\alpha}\,\x^\alpha)\quad\mapsto\quad
L_\y(f)\,=\,\sum_{\alpha}f_{\alpha}\,y_{\alpha},\quad f\in\R[\x].\]
With $d\in\N$, the $d$-moment matrix associated with $\y$
is the real symmetric matrix $\M_d(\y)$ with rows and columns indexed 
in $\N^n_d$, and defined by:
\begin{equation}
\label{moment}
\M_d(\y)(\alpha,\beta)\,:=\,L_\y(\x^{\alpha+\beta})\,=\,y_{\alpha+\beta},\qquad\forall\alpha,\beta\in\N^n_d.\end{equation}
Alternatively, $\M_d(\y)=\sum_{\alpha\in\N^n_{2d}}y_\alpha\B^0_\alpha$.
It is straightforward to check that
\[\left\{\,L_\y(g^2)\geq0\quad\forall g\in\R[\x]_d\,\right\}\quad\Leftrightarrow\quad\M_d(\y)\,\succeq\,0,\quad d=0,1,\ldots.\]
A sequence $\y=(y_\alpha)$ has a representing measure if there exists a finite Borel measure $\mu$ on $\R^n$, such that
$y_\alpha=\int \x^\alpha d\mu$ for every $\alpha\in\N^n$. Moreover, the measure $\mu$ is said to be {\it determinate} if
it is the unique such measure.
Notice that with the $\ell_\w$-norm on $\R[\x]$ is associated a dual norm $\Vert\cdot\Vert^*_\w$
on the dual space $\R[\x]^*$ of $\ell_\w$-continuous linear functionals on $\R[\x]$, by
$\Vert L_\y\Vert^*_\w =\sup\{\vert y_\alpha\vert/w_\alpha:\alpha\in\N^n\}$.

\subsection*{Localizing matrix}
With $\y$ as above, $J\subseteq\{1,\ldots,m\}$, and $g_J\in\R[\x]$ (with $g_J(\x)=\sum_\gamma g_{J\gamma}\,\x^\gamma$), the {\it localizing} matrix of order $d$ associated with $\y$ and $g_J$ is the real symmetric matrix $\M_d(g_J\,\y)$ with rows and columns indexed by $\N^n_d$, and whose entry $(\alpha,\beta)$ is just 
\begin{equation}
\label{local}
\M_d(\y)(g_J\,\y)(\alpha,\beta)\,:=\,L_\y(g_J(\x)\x^{\alpha+\beta})\,=\,
\sum_{\gamma}g_{J\gamma} \,y_{\alpha+\beta+\gamma},\quad\forall\alpha,\beta\in\N^n_d.\end{equation}
If $\B^J_\alpha\in\s^{s(d)}$ is defined by:
\begin{equation}
\label{calpha}
g_J(\x)\,\v_d(\x)\,\v_d(\x)^T\,=\,\sum_{\alpha\in\N^n_{2d+{\rm deg}\,g_J}}\B^J_\alpha\,\x^\alpha,\qquad\forall\x\in\R^n,\end{equation}
then $\M_d(g_J\,\y)=\sum_{\alpha\in\N^n_{2d+{\rm deg}g_J}}y_\alpha\,\B^J_\alpha$. Alternatively,
$\M_d(g_J\,\y)=\M_d(\z)$ where $\z=(z_\alpha)$, $\alpha\in\N^n$, with $z_\alpha=L_\y(g_J\,\x^\alpha)$.

\subsection*{Multivariate Carleman's condition}
Let $\y=(y_\alpha)$, $\alpha\in\N^n$, be such that $\M_d(\y)\succeq0$ for all $d\in\N$. If
for every $i=1,\ldots,n$,
\begin{equation}
\label{carleman}
\sum_{k=1}^\infty L_\y(x_i^{2k})^{-1/2k}\,=\,\infty,
\end{equation}
then $\y$ has a finite representing Borel measure $\mu$ on $\R^n$, which in addition, is determinate; see
e.g. Berg \cite{berg2}.

\subsection*{Closures}

For a set $A\subset\R[\x]$ we denote by $\overline{A}$ the closure of $A$ for the finest locally convex topology on 
$\R[\x]$ (treated as a real vector space). With this topology, every finite-dimensional subspace of $\R[\x]$ inherits the euclidean topology, so that 
$\overline{A}$ also denotes the usual euclidean closure of
a subset $A\subset\R[\x]_d$. Following Cimpric et al. \cite{cimpric} and
Kuhlmann et al. \cite{salma}, we also denote by $A^{\ddag}$ the set of all elements of $\R[\x]$
which are expressible as the limit of some sequence of elements of $A$, and so $A^{\ddag}$ is 
called the {\it sequential} closure of $A$, and clearly $A\subseteq A^{\ddag}\subseteq\overline{A}$.
If $A\subset\R[\x]$ is a convex cone 
\[A^{\ddag}\,=\,\{f\in\R[\x]\::\:\exists\,q\,\in\R[\x] \mbox{ s.t. } f+\epsilon\,q\in A,\quad\forall \epsilon>0\:\},\]
and in fact, $q$ can be chosen to be in $A$. Moreover, if $A$ has nonempty interior (equivalently, has an algebraic interior) then
$A^{\ddag}=\overline{A}$. 

\subsection*{Semidefinite programming} A semidefinite program is a convex (more precisely convex conic) optimization problem 
of the form 
\[\inf_{\X}\: \{\:\la \C,\X\ra\::\:\mathcal{A}\,\X\,=\b;\:\X\succeq0\:\},\]
for some real symmetric matrices $\C,\X\in\s^p$, vector $\b\in\R^m$,
and some linear mapping $\mathcal{A}:\s^p\to\R^m$. Semidefinite programming is a powerful technique 
of convex optimization, ubiquitous in many areas. A semidefinite program can be solved
efficiently and even in time polynomial in the input size of the problem, for fixed arbitrary precision.
For more details the interested reader is referred to e.g.
\cite{boyd}.

\subsection{A preliminary result of independent interest}

Recall that in a complete separable metric space $\X$, the support of a finite Borel measure $\mu$ 
(denoted $\supmu$) is the unique smallest closed set $\K\subseteq\X$ such that $\mu(\X\setminus\K)=0$.
\begin{thm}
\label{newlook}
Let $f\in\R[\x]$ and $\mu$ be a finite Borel measure with all moments 
$\y=(y_\alpha)$, $\alpha\in\N^n$, finite and such that for some $M>0$ and
all $k\in\N$ and all $i=1,\ldots,n$, $L_\y(x_i^{2k})\leq (2k)\l\,M$. Then:
\begin{eqnarray}
\label{mu}
\mbox{$f\geq0$ on $\supmu$ }&\Longleftrightarrow&\int h^2\,f\,d\mu\,\geq\,0\quad\forall h\in\R[\x]\\
\nonumber
&\Longleftrightarrow&\M_d(f\,\y)\,\succeq\,0,\qquad\forall d=0,1,\ldots
\end{eqnarray}
\end{thm}
\begin{proof}
The implication $\Rightarrow$ is clear. For the reverse implication,
consider the signed Borel measure
$\nu(B):=\int_Bfd\mu$, for all Borel sets $B\in\mathcal{B}$, and let $\z=(z_\alpha)$, $\alpha\in\N^n$, be its sequence of moments. 
By Lemma \ref{newcarleman}, the sequence $\z$ satisfies Carleman's condition (\ref{carleman}). Next, recalling
that $\M_k(f\,\y)=\M_k(\z)$ for every $k\in\N$,
\[\left(\int h^2f\,d\mu\,\geq\,0\quad\forall h\in\R[\x]\,\right)\,\Longleftrightarrow\quad\M_k(\z)\,\succeq\,0,\forall \,k\in\N.\]
This combined with the fact that $\z$ satisfies Carleman's condition yields that $\z$ is the moment sequence of a finite Borel measure
$\psi$ on $\R^n$, which in addition, is determinate. Therefore,
\begin{equation}
\label{auxx}
z_\alpha\,=\,\int \x^\alpha\,\underbrace{f(\x)d\mu(\x)}_{d\nu(\x)}\,=\,\int \x^\alpha\,d\psi(\x),\qquad\forall\alpha\in\N^n.\end{equation}
Let $\Gamma+:=\{\x:f(\x)\geq0\}$, $\Gamma^-:=\{\x:f(\x)<0\}$ and let
$\mu=\mu^++\mu^-$ with $\mu^+(B)=\mu(B\cap\Gamma^+)$,
$\mu^-(B)=\mu(B\cap\Gamma^-)$, for all $B\in\mathcal{B}$.
Similarly, let $\nu=\nu^+-\nu^-$ with $\nu^+$ and $\nu^-$ being the positive measures defined by
$\nu^+(B)=\int_Bfd\mu^+$ and $\nu^-(B)=-\int_Bfd\mu^-$, for all $B\in\mathcal{B}$.
Since $\mu^+,\mu^-\leq\mu$, one has
\[\int x_i^{2k}d\mu^+(\x)\,\leq\,\int x_i^{2k}d\mu(\x)\quad\mbox{and}\quad\int x_i^{2k}d\mu^-(\x)\,\leq\,
\int x_i^{2k}d\mu(\x),\]
for all $i=1,\ldots,n$ and all $k\in\N$. Therefore, again by  Lemma \ref{newcarleman}, both $\nu^+$ and $\nu^-$ satisfy Carleman's condition (\ref{carleman})
so that both are determinate. On the other hand, (\ref{auxx}) can be rewritten,
\[\int \x^\alpha\,d\nu^+(\x)\,=\,\int \x^\alpha\,d\nu^-(\x)+
\int \x^\alpha\,d\psi(\x),\qquad\forall\alpha\in\N^n,\]
and so $\nu^+=\nu^-+\psi$ because $\nu^+$ and $\nu^-+\psi$ are determinate.
But then $0=\nu^+(\Gamma^-)\geq\nu^-(\Gamma^-)$ implies that $\nu^-=0$, i.e.,
$\nu=\nu^+=\psi$, and so the signed Borel measure $\nu$ is in fact a positive measure.
This in turn implies that $f\geq0$ for all $\x\in\supmu\setminus G$, where $G\subset\supmu$ and
$\mu(G)=0$. Notice that by minimality of the support ,
$\overline{\supmu\setminus G}=\supmu$. Hence let $\x\in\supmu$ be fixed, arbitrary. 
As $\overline{\supmu\setminus G}=\supmu$, there is sequence
$(\x_\ell)\subset \supmu\setminus G$ such that $\x_\ell\to\x$ as $\ell\to\infty$, and $f(\x_\ell)\geq0$ for all $\ell$.
But then continuity of $f$ yields that $f(\x)\geq0$.
\end{proof}
Interestingly, as we next see, Theorem \ref{newlook} yields alternative characterizations of the cone
$\psd(\bS)$ for an arbitrary closed set $\bS\subset\R^n$.

For a finite Borel measure $\mu$ (with all moments finite) and a polynomial $f\in\R[\x]$, let $\mu_f$ be the finite signed Borel measure defined by $\mu_f(B):=\int_{B}fd\mu$ for all $B\in\mathcal{B}$.
Let $\Theta_\mu:=\{\mu_\sigma:\sigma\in\Sigma[\x]\}$, i.e., $\Theta_\mu$ is the set of finite Borel measures absolutely continuous with respect to $\mu$, and whose density (or Radon Nikodym derivative) is a sum of squares of polynomials. 

Let $\Sigma[\x]^*\subset\R[\x]^*$ be the dual cone of the cone of $\Sigma[\x]$, i.e., the set of linear forms $L_\y$
on $\R[\x]$ such that $L_\y(h^2)\geq0$ for all $h\in\R[\x]$, and similarly,
let $\Theta_\mu^*\subset\R[\x]$ be the dual cone of $\Theta_\mu$, i.e., 
$\Theta_\mu^*:=\{h\in\R[\x]\,:\,\int hd\nu\geq0,\:\forall \nu\in\Theta_\mu\}$.
\begin{cor}
\label{consequences}
Let $\bS\subseteq\R^n$ be an 
arbitrary closed set and let $\mu$ be any finite Borel measure such that
$\supmu=\bS$ and $\x\mapsto\exp(\vert x_i\vert)$ is $\mu$-integrable for all $i=1,\ldots,n$.
Then with $f\in\R[\x]$, 
\begin{eqnarray}
\label{cons-1}
f\,\in\,\psd(\bS)&\Longleftrightarrow& \mu_f\in\Sigma[\x]^*\\
\label{cons-2}
\psd(\bS)&=&\Theta_\mu^*.
\end{eqnarray}
\end{cor}
\begin{proof}
Let $\y=(y_\alpha)$, $\alpha\in\N^n$, be the moment sequence of $\mu$.
Observe that for every $i=1,\ldots,n$, and all $k\in\N$,
\[L_\y\left(\frac{x_i^{2k}}{(2k)\l}\right)\,\leq\,\int_{\bS} \exp{(\vert x_i\vert)}d\mu(\x)
\,\leq\,M,\]
for some $M>0$. Moreover, $\supmu=\bS$ and so by Theorem \ref{newlook},
$f\geq0$ on $\bS$ if and only if $\int_{\bS} h^2fd\mu\geq0$ for all $h\in\R[\x]$. Equivalently, if and only if
$\int_{\bS} \sigma d\mu_f\geq0$ for all $\sigma\in\Sigma[\x]$ (which yields (\ref{cons-1})), or if and only if
$\int_{\bS}  fd \mu_\sigma\geq0$ for all $\sigma\in\Sigma[\x]$, which yields (\ref{cons-2}).
\end{proof}

\section{The $\K$-moment problem for $\ell_\w$-continuous linear functionals}

We first show that $Q(g)$ (and $P(g)$) solve the $\K$-moment problem
for $\ell_\w$-continuous linear functionals.

When equipped with the $\ell_\w$-norm, we may and will identify $\R[\x]$ 
as the subspace of sequences with finite support, 
in the Banach space of real infinite sequences $\f=(f_\alpha)$, $\alpha\in\N^n$, that are
$\w$-summable, i.e., such that $\sum_\alpha w_\alpha\vert f_\alpha\vert<+\infty$.

\begin{prop}
\label{dual-of-rx}
The dual of $(\R[\x],\Vert\cdot\Vert_\w)$ is the space $(\R[\x]^*,\Vert \cdot\Vert_\w^*)$ of 
linear functionals $L_\y$ associated with the sequence $\y=(y_\alpha)$, $\alpha\in\N^n$, which satisfy
$\Vert L_\y\Vert_\w^*<\infty$, where $\Vert L_\y\Vert^*_\w:=\sup\{\vert y_\alpha\vert/w_\alpha:\alpha\in\N^n\}$. 
\end{prop}
\begin{proof}
If $L_\y\in \R[\x]^*$ satisfies $\Vert L_\y\Vert_\w^*<\infty$, then
\[\vert L_\y(f)\vert \,\leq\,\sum_{\alpha}\vert f_\alpha\vert \,w_\alpha\,\frac{y_\alpha}{w_\alpha}\,\leq\,\Vert f\Vert_\w\,\Vert L_\y\Vert^*_\w,\]
and so $L_\y$ is bounded, hence $\ell_\w$-continuous. Conversely, if $L_\y$ is $\ell_\w$-continuous, then 
consider the sequence of polynomials $(f_\alpha)\subset\R[\x]$ with
$\x\mapsto f_\alpha(\x):=\x^\alpha/w_\alpha$ for every $\alpha\in\N^n$.
Then $\Vert f_\alpha\Vert _\w=1$ for every $\alpha\in\N^n$,
and $L_\y(f_\alpha)=y_\alpha/w_\alpha$  for all $\alpha\in\N^n$. And so,
if $L_\y$ is $\ell_\w$-continuous then $\sup\{\vert y_\alpha\vert/w_\alpha:\alpha\in\N^n\}<+\infty$.
\end{proof}
\begin{thm}
\label{k-moment}
Let $\K\subseteq\R^n$ be as in (\ref{setk}) and let
$\y=(y_\alpha)$, $\alpha\in\N$, be a given real sequence such that
$L_\y$ is $\ell_\w$-continuous.
Then $\y$ has a finite representing Borel measure $\mu$ on $\K$ if and only if
$L_\y$ is nonnegative on $Q(g)$. Equivalently if and only if:
\begin{equation}
\label{carac}
L_\y(h^2g_j)\geq0\quad\forall\,h\in\R[\x], \:j=0,\ldots,m;\quad \sup_{\alpha\in\N^n} \frac{\vert y_\alpha\vert}{w_\alpha}\,\leq\,M\mbox{ for some $M$.}\end{equation}
\end{thm}
\begin{proof}
The necessity is clear. Indeed, if $\y$ has a representing measure
$\mu$ on $\K$ then $L_\y(h^2g_j)=\int_\K h^2g_jd\mu\geq0$ for all $h\in\R[\x]$
and all $j=0,1,\ldots,m$; and so $L_\y(f)\geq0$ for all $f\in Q(g)$.
Moreover, $\Vert L_\y\Vert^*_\w<\infty$ because $L_\y$ is $\ell_\w$-continuous; hence (\ref{carac}) holds.

Sufficiency. Suppose that $L_\y$ is a non trivial $\ell_\w$-continuous linear functional, nonnegative on $Q(g)$, i.e.,
suppose that (\ref{carac}) holds. In particular,
$L_\y(x_i^{2k})\leq M(2k)\l$ for every $k\in\N$ and every $i=1,\ldots,n$. Therefore,
$\y$ satisfies Carleman's condition (\ref{carleman}) and since $\M_k(\y)\succeq0$ for all $k\in\N$,
$\y$ has a representing finite Borel measure $\mu$ on $\R^n$, which in addition, is determinate.
Next, using $L_\y(h^2g_j)\geq0$ for all $h\in\R[\x]$, and invoking Theorem \ref{newlook}, one may conclude that
$g_j\geq0$ on $\supmu$, for every $j=1,\ldots,m$. Hence $\supmu\subseteq\K$.
\end{proof}
Theorem \ref{k-moment} states that $Q(g)$ solves the $\K$-moment problem 
for $\ell_\w$-continuous functionals. Of course, Theorem \ref{k-moment} is also true if one replaces the quadratic module
$Q(g)$ with the preordering $P(g)$. 

\subsection*{The $\ell_\w$-closure of $Q(g)$ and $P(g)$}
Observe that $\psd(\K)$ is $\ell_\w$-closed. To see this, notice that with every 
every $\x\in\K$ is associated the Dirac measure 
$\delta_\x$, whose associated sequence $\y=(\x^\alpha)$, $\alpha\in\N^n$, is such that 
$L_\y$ is $\ell_\w$-continuous.
Indeed, let $a:=\max_i\vert \x_i\vert$ so that $\vert\x^\alpha\vert\leq a^{\vert\alpha\vert}$, and let $M:=\exp(a)$.
\[M^{-1}\,\vert y_\alpha\vert \,=\,\exp(a)^{-1} \vert\x^\alpha\vert\,\leq\,\exp(a)^{-1}\,a^{\vert \alpha\vert}\,<\,\alpha\l
\,\leq\,(2\lceil \vert\alpha\vert/2\rceil)\l\,=\,w_\alpha,\]
and so $\Vert L_\y\Vert^*_\w<M$, i..e., $L_\y$ is $\ell_\w$-continuous.
Therefore, let $(f_n)\subset\psd(\K)$ be such that $\Vert f_n-f\Vert_\w\to 0$ as $n\to\infty$.
As $L_\y$ is $\ell_\w$-continuous one must have $0\leq \lim_{n\to\infty}L_\y(f_n)=L_\y(f)=f(\x)$.
As $\x\in\K$ was arbitrary, $f\in\psd(\K)$.

\begin{thm}
\label{thmain-4}
Let $\K\subseteq\R^n$ be as in (\ref{setk}) and recall that $\psd(\K):=\{f\in\R[\x]\,:\,f\geq0\mbox{ on }\K\}$.
Then ${\rm cl}_\w(P(g))={\rm cl}_\w(Q(g))=\psd(\K)$.
\end{thm}

\begin{proof}
As $\psd(\K)$ is $\ell_\w$-closed and $Q(g)\subset \psd(\K)$, ${\rm cl}_\w(Q(g))\subseteq\psd(\K)$, and so we only have
to prove the reverse inclusion.
Let $f\not\in{\rm cl}_\w(Q(g))$. Since $Q(g)$ is a convex cone, by the Hahn-Banach separation theorem there exists
an $\ell_\w$-continuous linear functional $L_\y$ that strictly separates $f$ from
${\rm cl}_\w(Q(g))$. That is, there exists $\y\in\N^n$ such that $L_\y$ is $\ell_\w$-continuous,
$L_\y(f)<0$ and $L_\y(h)\geq0$ for all
$h\in{\rm cl}_\w(Q(g))$. By Theorem \ref{k-moment}, such a $\y$ has a representing finite Borel measure $\mu$ on $\K$, and so $L_\y(f)=\int_\K fd\mu<0$ yields that necessarily $f(\x_0)<0$ for some $\x_0\in\K$. 
Hence $\psd(\K)\subseteq {\rm cl}_\w(Q(g))$, which in turn yields the desired result.
\end{proof}
For instance, from Berg \cite{berg2}, the $\ell_1$-closure of $\Sigma[\x]$ is $\psd([-1,1]^n)$. On the other hand,
its $\ell_\w$-closure is now $\psd(\R^n)$, which is what we really want.

It is worth mentioning that Theorem \ref{k-moment} and \ref{thmain-4} can be extended to any preordering or 
quadratic module (finitely generated or not). The proof is the same. Any closed set $\K\subset\R^n$ 
may be represented as the non-negativity set 
of such a quadratic module (taking generators of the form
$\x\mapsto g(\x):=\sum_{i=1}^n(x_i-a_i)-r^2$, for suitable $\a\in\R^n$ and $r>0$). However, 
for reasons that become obvious in the next section, the focus of the present paper is on the finitely generated case. 

\section{Canonical $\ell_\w$-projections onto $P_d(g)$ and $Q_d(g)$}

As we next see, the $\ell_\w$- and $\ell_1$-norm have the nice feature 
that one may find particular (canonical) projections onto various truncations of $P(g)$ and $Q(g)$ with a particularly simple expression.
For $\ell_\w$- and $\ell_1$- projections to be well-defined we assume that $\K$ has a nonempty interior.

We first provide an explicit form of  {\it canonical} $\ell_1$- and $\ell_\w$-projection of any given polynomial $f$ onto $P_d(g)$ and $Q_d(g)$ respectively, and
analyze their limit as $d\to\infty$. Then we will consider the projections onto $P_d(g)\cap\R[\x]_s$ for fixed
$s,d\in\N$, which, letting $s\to\infty$, will permit to obtain the projection of $f$
onto $\overline{P(g)\cap\R[\x]_d}$, and so to also characterize the sequential closure $P(g)^{\ddag}$.

As in the previous section, for a polynomial in $\R[\x]_t$ we use indifferently the notation
$h$ for both the polynomial and its vector of coefficients $\h\in\R^{s(t)}$. The context 
will make clear which one of the two is concerned.

Let $\K\subseteq\R^n$ be as in (\ref{setk}) and consider the following optimization problem:
\begin{equation}
\label{inv-1}
p_{d\w}:=\displaystyle\inf_h\,\{\:\Vert f-h\Vert_\w\::\: h\in P_d(g)\:\}.
\end{equation}
That is, one searches for the best $\ell_\w$-approximation of $f$ by an element $h^*$ of $P_d(g)$, or equivalently,
an $\ell_\w$-projection of $f$ onto $P_d(g)$. In general, such a best approximation $h^*\in P_d(g)$ is not unique\footnote{The following example was kindly provided by an anonymous referee:  Let $n=2,d=1$, $\x\mapsto f(\x):=-2x_1x_2$,
and let $C\subset\R[\x]_{2d}$ be the cone of sums of squares of linear polynomials. Then any polynomial $\x\mapsto p_\lambda(\x):=\lambda (x_1-x_2)^2$ with $\lambda\in [0,1]$, is an $\ell_\w$-projection (and an $\ell_1$-projection) of $f$ onto $C$ and $\Vert f-p_\lambda\Vert_\w=2$.}
but we provide a {\it canonical} solution with a very simple expression.
Of course, and even though (\ref{inv-1}) is well defined for an arbitrary $f\in\R[\x]$,
such a problem is of particular interest when $f$ is nonnegative on $\K$ but not necessarily an element of $P(g)$.

\begin{thm}
\label{thmain}
Let $\K\subseteq\R^n$ in (\ref{setk}) be with nonempty interior.
Let $f\in\R[\x]$ and let $2d\geq {\rm deg}\,f$.
There is an $\ell_\w$-projection of $f$ onto $P_d(g)$ which is a polynomial $g^{P\w}_f\in\R[\x]_{2d}$
of the form:
\begin{equation}
\label{best-p}
\x\,\mapsto\: g^{P\w}_{f}(\x)\,:=\,f(\x)+\,\left(\lambda^{P\w}_{0}+\sum_{i=1}^n\sum_{k=1}^d 
\lambda^{P\w}_{ik}\,\frac{x_i^{2k}}{(2k)\l}\,\right),
\end{equation}
where the nonnegative vector $\lambda^{P\w}\in\R^{nd+1}$
is an optimal solution of the semidefinite program:
\begin{equation}
\label{dual3-p}
\displaystyle\inf_{\lambda\geq0}\:\left\{\displaystyle\lambda_0+\sum_{i=1}^n\sum_{k=1}^d\lambda_{ik}\::\:
f+\lambda_0+\displaystyle\sum_{i=1}^n\sum_{k=1}^d\lambda_i\:\frac{x_{ik}^{2k}}{(2k)\l}\in P_d(g)\:\right\},
\end{equation}
and  $p_{d\w}=\Vert f-g^{P\w}_f\Vert_\w=\displaystyle\lambda^{P\w}_0+\sum_{i=1}^n\sum_{k=1}^d\lambda^{P\w}_{ik}$.
\end{thm}
\begin{proof}
Consider $f$ as an element of $\R[\x]_{2d}$ by setting $f_\alpha=0$ whenever
$\vert\alpha\vert >{\rm deg}\,f$ (where $\vert\alpha\vert=\sum\alpha_i$),
and rewrite (\ref{inv-1}) as the semidefinite program:

 \begin{equation}
\label{inv-3}
\begin{array}{rl}
p_{d\w}:=\displaystyle\inf_{\lambda,\X_J,h}&\displaystyle\sum_{\alpha\in\N^n_{2d}}w_\alpha\,\lambda_\alpha\\
\mbox{s.t.}& \lambda_\alpha +h_\alpha\geq f_\alpha,\quad\forall \alpha\in\N^n_{2d}\\
& \lambda_\alpha -h_\alpha\geq -f_\alpha,\quad\forall \alpha\in\N^n_{2d}\\
&h_\alpha-\displaystyle\sum_{J\subseteq\{1,\ldots,m\}}^m\la\X_J,\B^J_\alpha\ra =0,\quad\forall \alpha\in\N^n_{2d}\\
&\lambda\geq0;\:h\in\R[\x]_{2d};\:\X_J\succeq0,\:\forall J\subseteq\{1,\ldots,m\}.
\end{array}
\end{equation}
The dual semidefinite program of (\ref{inv-3}) reads:

\begin{equation}
\label{dual}
\left\{\begin{array}{rrll}
\displaystyle\sup_{\u,\v\geq 0,\y}&\displaystyle
\sum_{\alpha\in\N^n_{d}}f_\alpha(u_\alpha-v_\alpha)&&\\%\:(=L_\y(f))\\
\mbox{s.t.}&u_\alpha+v_\alpha&\leq \,w_\alpha&\forall\alpha\in\N^n_{2d}\\
&u_\alpha-v_\alpha+y_\alpha &\,=\,0&\forall\alpha\in\N^n_{2d},\\
&\M_d(g_J\,\y)&\succeq\,0,&\forall J\subseteq\{1,\ldots,m\},
\end{array}\right.
\end{equation}
or, equivalently,
\begin{equation}
\label{dual1}
\left\{\begin{array}{ll}
\displaystyle\sup_{\y}&-L_\y(f)\\
\mbox{s.t.}&\M_d(g_J\,\y)\,\succeq\,0,\qquad\forall J\subseteq\{1,\ldots,m\}\\
&\vert y_\alpha\vert \,\leq\,w_\alpha,\quad\forall\alpha\in\N^n_{2d}.
\end{array}\right.
\end{equation}
The semidefinite program (\ref{dual1}) has an optimal solution $\y^*$ because the feasible set is nonempty and compact.
In addition, 
let $\y=(y_\alpha)$ be the moment sequence of
the finite Borel measure $\mu(B)=\int_{\K\cap B}\e^{-\Vert\x\Vert^2}d\x$, for all $B\in\mathcal{B}$, scaled so that $\vert y_\alpha\vert <w_\alpha$ for all $\alpha\in\N^n_{2d}$.
Then $(\y,\u,\v)$ with $\u=-\min[\y,0]$ and $\v=\max[\y,0]$, is strictly feasible in (\ref{dual}). Indeed,
as $\K$ has nonempty interior, then necessarily $\M_d(g_J\,\y)\succ 0$ for all $J\subseteq\{1,\ldots,m\}$,
and so Slater's condition\footnote{Slater's condition holds for the conic optimization problem $\P:\:\inf_\x \{\c'\x\,:\, \A\x=\b;\:\x\in\K\}$, where $\K\subset\R^n$ is a convex cone and $\A\in\R^{p\times n}$, if
there exists a feasible solution $\x_0\in{\rm int}\,\K$.
In this case, there is no duality gap
between $\P$ and its dual $\P^*:\:\sup_\z\{\b'\z\,:\,\c-\A'\z\in\K^*\}$. 
In addition, if the optimal value is bounded then $\P^*$ has an optimal solution.}
holds for (\ref{dual}). 
Therefore, by a standard duality result in convex optimization, there is no duality gap
between (\ref{inv-3}) and (\ref{dual}) (or (\ref{dual1})), and (\ref{inv-3}) has an optimal solution $(\lambda^*,(\X^*_j),g^P_f)$.
Hence $p_{d\w}=-L_{\y^*}(f)$ for any optimal solution $\y^*$ of (\ref{dual1}).

Next, recall that with $J:=\emptyset$, $\M_d(g_\emptyset\,\y)=\M_d(\y)$.
Moreover, $\M_d(\y)\succeq0$ implies $\M_k(\y)\succeq0$ for all $k\leq d$.
By \cite[Lemma 1]{arch}, $\M_k(\y)\succeq0$ implies that
$\vert y_\alpha\vert\leq\max[L_\y(1),\max_iL_\y(x_i^{2k})]$, for every $\alpha\in\N^n_{2k}$,
and all $k\leq d$. Therefore,
(\ref{dual1}) has the equivalent formulation
\begin{equation}
\label{dual2}
\left\{\begin{array}{rrl}
p_d=-\displaystyle\inf_{\y}&L_\y(f))&\\
\mbox{s.t.}&\M_d(g_J\,\y)&\succeq\,0,\quad \forall J\subseteq\{1,\ldots,m\}\\
&L_\y(1)&\leq\,1\\
&L_\y(x_i^{2k})&\leq\,(2k)\l,\quad k=1,\ldots,d;\:  i=1,\ldots,n.
\end{array}\right.
\end{equation}
Indeed, any feasible solution of (\ref{dual2}) 
satisfies 
\[\vert y_\alpha\vert\leq\max[L_\y(1),\max_iL_\y(x_i^{2k})]\leq (2k)\l\,=\,w_\alpha,\]
for every $\alpha$ with $\vert\alpha\vert =2k$ and $2k-1$, $k=1,\ldots,d$. The dual of (\ref{dual2})
is exactly the semidefinite program (\ref{dual3-p}).
Again Slater's condition holds for (\ref{dual2}) and it has an optimal solution $\y^*$.
Therefore (\ref{dual3-p}) also has an optimal solution $\lambda^{P\w}\in\R^{nd+1}_+$ with
$p_{d\w}=\lambda_0^{\P\w}+\sum_{i=1}^n\sum_{k=1}^d\lambda_{ik}^{P\w}$, which is the desired result.
\end{proof}
The polynomial $g^{P\w}_{f}\in P_d(g)$ in (\ref{best-p}) is what we call the canonical $\ell_\w$-projection of $f$ onto $P_d(g)$.

Of course, all statements in Theorem \ref{thmain} remain true if one replaces $P_d(g)$ with $Q_d(g)$.
Moreover, if $w_\alpha=1$ for all $\alpha$ (and so
$\ell_\w$ is now the usual $\ell_1$-norm) the polynomial $g_f^{P\w}$ in (\ref{best-p}) simplifies and is of the form:
\begin{equation}
\label{simplify}
\x\,\mapsto\: g^P_f(\x)\,:=\,f(\x)+\,(\lambda^P_0+\sum_{i=1}^n \lambda^P_i\,x_i^{2d}),\end{equation}
for some nonnegative vector $\lambda^P\in\R^{n+1}$.
If $\K=\R^n$ then $g^P_f$ is the canonical $\ell_1$-projection of $f$ onto the cone of s.o.s. polynomials., as illustrated in the following example.
\begin{ex}
{\rm Let $n=2$ and $\K=\R^2$, in which case
$P_d(g)=Q_d(g)=\Sigma[\x]_d$ for all $d\in\N$. Consider the Motzkin-like polynomial\footnote{Computation was made by running the GloptiPoly software \cite{didier}
dedicated to solving the generalized problem of moments.}  $\x\mapsto f(\x)=x_1^2x_2^2(x_1^2+x_2^2-1)+1/27$ of degree $6$, which is nonnegative but not  
a s.o.s., and with a global minimum $f^*=0$ attained at  4 global minimizers $\x^*=(\pm (1/3)^{1/2}, \pm (1/3)^{1/2})$.
The results are displayed in Table \ref{tab1} for $d=3,4,5$.

\begin{table}[ht]
\label{tab1}
\begin{center}
\begin{tabular}{|| l | l | l ||}
\hline
$d$ & $\lambda^*$  & $p_d$ \\
\hline
\hline
$3$ &  $\approx 10^{-3}\,(5.445,  5.367 , 5.367)$ & $\approx 1.6\, 10^{-2}$\\
$4$& $\approx 10^{-4}\,(2.4 ,  9.36 , 9.36)$ &$\approx 2.\,10^{-3}$\\
$5$& $\approx 10^{-5}\,(0.04 ,  4.34, 4.34)$ &$\approx 8.\,10^{-5}$\\
\hline 
\end{tabular}
\end{center}
\caption{Best $\ell_1$-approximation for the Motzkin polynomial.}
\end{table}
}
\end{ex}

\subsection{Canonical $\ell_\w$-projection onto $P(g)_t\cap\R[\x]_{2d}$ and $Q(g)_t\cap\R[\x]_{2d}$}

We now consider $\ell_\w$-projection of $f$ onto $P_t(g)\cap\R[\x]_{2d}$ for given
integers $d,t\in\N$, i.e.,
\begin{equation}
\label{inv-111}
p^{d\w}_t:=\displaystyle\inf_g\,\{\:\Vert f-g\Vert_\w\::\: g\in P_t(g)\,\cap\,\R[\x]_{2d}\:\}.
\end{equation}
In other words, we are interested in searching for
a polynomial of degree $2d$ in $P_t(g)$ which is the closest to $f$ for the $\ell_\w$-norm. 
For instance, if $2d={\rm deg}\,f$, one wishes to find an $\ell_\w$-projection onto $P_t(g)$
of same degree as $f$. One may also consider the analogue problem with
the quadratic module, i.e., an $\ell_\w$-projection onto $Q_t(g)\cap\R[\x]_{2d}$.
We also analyze the limit as $t\to\infty$. 

\begin{thm}
\label{thmain-2}
Let $\K\subseteq\R^n$ in (\ref{setk}) be with nonempty interior and let $d\in\N$.
Let $f\in\R[\x]$ and let $2t\geq \max[2d,{\rm deg}\,f]$.
There is an $\ell_\w$-projection of $f$ onto $P_t(g)\cap\R[\x]_{2d}$ 
which is a polynomial $g^{P\w}_f\in\R[\x]_{2d}$ of the form:
\begin{equation}
\label{best-p-2w}
\x\,\mapsto\: g^{P\w}_f(\x)\,=\,f(\x)+\,\left(\lambda^{P\w}_0+\sum_{i=1}^n\sum_{k=1}^d \lambda^{P\w}_{ik}\,\frac{x_i^{2d}}{(2k)\l}\,\right),
\end{equation}
where the nonnegative vector $\lambda^{P\w}\in\R^{nd+1}$
is an optimal solution of the semidefinite program:
\begin{equation}
\label{dual3-p-2w}
p^{d\w}_t\,=\,\displaystyle\inf_{\lambda\geq0}\:\left\{\displaystyle\lambda_0+\sum_{i=1}^n\sum_{k=1}^d\lambda_{ik}\::\:
f+\lambda_0+\displaystyle\sum_{i=1}^n\sum_{k=1}^d\lambda_{ik}\,\frac{x_i^{2d}}{(2k)\l}\in P_t(g)\:\right\},
\end{equation}
and $p^{d\w}_t=\Vert f-g^P
{P\w}_f\Vert_\w=\lambda^{P\w}_0+\displaystyle\sum_{i=1}^n\sum_{k=1}^d\lambda^{P\w}_{ik}$.
\end{thm}
\begin{proof}
The proof is almost a verbatim copy of that of Theorem \ref{thmain} with a slight
difference. For instance, (\ref{inv-3}) is now replaced with
 \begin{equation}
\label{inv-33}
\begin{array}{rl}
p^{d\w}_t:=\displaystyle\inf_{\lambda,\X_J,h}&\displaystyle\sum_{\alpha\in\N^n_{2d}}w_\alpha\,\lambda_\alpha\\
\mbox{s.t.}& \lambda_\alpha +h_\alpha\geq f_\alpha,\quad\forall \alpha\in\N^n_{2d}\\
& \lambda_\alpha -h_\alpha\geq -f_\alpha,\quad\forall \alpha\in\N^n_{2d}\\
&h_\alpha-\displaystyle\sum_{J\subseteq\{1,\ldots,m\}}^m\la\X_J,\B^J_\alpha\ra =0,\quad\forall \alpha\in\N^n_{2t}\\
&\lambda\geq0;\:h\in\R[\x]_{2d};\:\X_J\succeq0,\:\forall J\subseteq\{1,\ldots,m\}.
\end{array}
\end{equation}
(where $h_\alpha=0$ whenever $\vert\alpha\vert>2d$) and the dual (\ref{dual1}) now reads
\begin{equation}
\label{dual11}
\left\{\begin{array}{ll}
\displaystyle\sup_{\y}&-L_\y(f)\\
\mbox{s.t.}&\M_t(g_J\,\y)\,\succeq\,0,\qquad\forall J\subseteq\{1,\ldots,m\}\\
&\vert y_\alpha\vert \,\leq\,w_\alpha,\quad\forall\alpha\in\N^n_{2d}.
\end{array}\right.
\end{equation}
Again with exactly the same arguments, (\ref{inv-33}) has a feasible solution and
(\ref{dual11}) has a strictly feasible solution $\y$, and so Slater's condition holds for (\ref{dual11}),
which in turn implies that there is no duality gap between (\ref{inv-33}) and (\ref{dual11}), and 
(\ref{inv-33}) has an optimal solution $(\lambda,(\X_J),g_f^{P\w})$. 
However (and this is the only difference with the proof of Theorem \ref{thmain}) now one cannot guarantee any more that (\ref{dual11}) has an optimal solution because $\vert y_\alpha\vert\leq w_\alpha$ only for
$\alpha\in\N^n_{2d}$ and not for all $\alpha\in\N^n_{2t}$.
\end{proof}
As before, we call $g^{P\w}_f$ in (\ref{best-p-2w}) the canonical $\ell_\w$-projection of $f$ onto 
$P_t(g)\cap\R[\x]_{2d}$.
Of course, an analogue of Theorem \ref{thmain-2} (with obvious {\it ad hoc} adjustments) holds for the
canonical $\ell_\w$-projection $g^{Q\w}_f$ onto $Q_t(g)\cap\R[\x]_{2d}$. Also and again, if $w_\alpha=1$
for all $\alpha\in\N^n_{2d}$,
then the polynomial $g^{P\w}_f$ in (\ref{best-p-2w}) simplifies to the form in (\ref{simplify}).

We next analyze the behavior of $g_f^{P\w}$ as $t\to\infty$ to obtain the canonical $\ell_\w$-projection of $f$ onto
$\overline{P(g)\cap\R[\x]_{2d}}$. Recall that
\[\overline{P(g)\cap\R[\x]_{2d}}\,=\,\overline{\left(\bigcup_{t\geq0}P_t(g)\right)\cap\R[\x]_{2d}}\,=\,\overline{\bigcup_{t\geq0}P_t^d(g)}.\]

\begin{cor}
\label{cor-2}
Let $\K\subseteq\R^n$ be as in (\ref{setk}) and with nonempty interior, $f\in\R[\x]_{2d}$, and let $g^{P\w}_f(t)\in
P_t(g)\cap\R[\x]_{2d}$ be an optimal solution of (\ref{inv-111}).
Then there is an $\ell_\w$-projection of $f$ onto $\overline{P(g)\cap\R[\x]_{2d}}$
which is a polynomial $g^\w_f\in\R[\x]_{2d}$ of the form
\begin{equation}
\label{best-p-22}
\x\,\mapsto\: g^\w_f(\x)\,=\,f(\x)+\,\left(\lambda^*_0+\sum_{i=1}^n \sum_{k=1}^d\lambda^*_{ik}\,x_i^{2d}\,\right),
\end{equation}
for some nonnegative vector $\lambda^*\in\R^{dn+1}$.
In particular, if $\K$ is compact  and $f\geq0$ on $\K$ then $\lambda^*=0$ and $g^\w_f=f$.
\end{cor}
\begin{proof}
Let $(p^{d\w}_t)$, $t\in\N$, be the sequence of optimal values of (\ref{inv-111}), which is nonnegative and monotone non increasing. Therefore it converges to some nonnegative value $p^{d\w}\geq0$.
For every $t\in\N$, (\ref{inv-111}) has an optimal solution $g^{P\w}_f(t)$ of the form
\begin{equation}
\label{aux}
\x\mapsto g_f^{P\w}(t)(\x)\,=\,f(\x)+\lambda_0^{P\w}(t)+\sum_{i=1}^n\sum_{k=1}^d\lambda^{P\w}_{ik}(t)\frac{x_i^{2k}}{2k\l},\qquad\forall\x\in\R^n,\end{equation}
with $\lambda^{P\w}(t)\geq0$ and $\sum_{i,k}\lambda^{P\w}_{ik}(t)=p^{d\w}_t\leq p^{d\w}_{t_0}$.
Hence there is a subsequence $(t_j)$, $j\in\N$, and some nonnegative vector $\lambda^*\in\R^{nd+1}_+$
such that $\lambda^{P\w}(t_j)\to\lambda^*$ as $j\to\infty$. In addition, 
\[p^{d\w}\,=\,\lim_{j\to\infty}p^{d\w}_{t_j}\,=\,\lim_{j\to\infty}\lambda_0^{P\w}(t_j)+
\sum_{i=1}^n\sum_{k=1}^d\lambda_{ik}^{P\w}(t_j)
\,=\,\lambda_0^*+\sum_{i=1}^n\sum_{k=1}^d\lambda_{ik}^*.\]
Hence $g^{P\w}_f(t_j)\to g^\w_f\in\R[\x]_{2d}$ as $j\to\infty$ where $g^\w_f$ is as in (\ref{best-p-22}).
And of course, as $\Vert g^{P\w}_f(t)-g^\w_f\Vert_\w\to0$, $g^\w_f$ is in the closure 
$\overline{P(g)\cap\R[\x]_{2d}}$ of $P(g)\cap\R[\x]_{2d}$.

Next, suppose that there exists $h\in\overline{P(g)\cap\R[\x]_{2d}}$ with $\Vert f-h\Vert_\w<\Vert f-g^\w_f\Vert_\w$.
Then there exists a sequence $(h_t)\subset\R[\x]_{2d}$, $t\in\N$, with
$h_t\in P^d_t(g)$ such that $\Vert h_t-h\Vert_\w\to 0$ as $t\to\infty$. But then
\[\Vert f-h\Vert_\w\,=\,\Vert f-h_t+h_t-h\Vert_\w\,\geq\,\underbrace{\Vert f-h_t\Vert_\w}_{\geq\Vert f-g^{P\w}_f(t)\Vert_\w}-
\underbrace{\Vert h_t-h\Vert_\w}_{\to0 \mbox{ as }t\to\infty},\qquad\forall\,t,\]
and so taking limit as $t\to\infty$ yields the contradiction
\[\Vert f-g^\w_f\Vert_\w\,>\,\Vert f-h\Vert_\w\,\geq\,\lim_{t\to\infty}\Vert f-g^{P\w}_f(t)\Vert_\w\,=\,\Vert f-g^\w_f\Vert_\w.\]
Therefore $\Vert f-g^\w_f\Vert_\w=\min_{h}\,\{\Vert f-h\Vert_\w\,:\,h\in\overline{P(g)\cap\R[\x]_{2d}}\}$.

The last statement (when $\K$ is compact) follows from Schm\"udgen's Positivstellensatz \cite{schmudgen} which implies that if $f$ is nonnegative on $\K$ then $f+\epsilon\in P(g)$ 
for every $\epsilon>0$.
\end{proof}
Of course there is an analogue of Corollary \ref{cor-2} with $Q(g)$ in lieu of $P(g)$. The only change is concerned with the last statement where one needs that $Q(g)$ is Archimedean. And also, if $w_\alpha=1$ for every $\alpha\in\N^n$,
then $g^\w_f$ in (\ref{best-p-22}) simplifies to the form (\ref{simplify}).

\subsection{The sequential closures of $P(g)$ and $Q(g)$}

Recall that for any convex cone $A\subset\R[\x]$
\begin{equation}
\label{seq-closure}
A^{\ddag}\,=\,\{f\in\R[\x]\::\:\exists\,q\,\in\R[\x] \mbox{ s.t. } f+\epsilon\,q\in A,\quad\forall \epsilon>0\:\}.\end{equation}
We have seen that
$P(g)\subset P(g)^{\ddag}\subseteq\overline{P(g)}$,
where $\overline{A}$ denotes the closure of $A$ for the finest locally convex topology.
Therefore, it is of particular interest to describe $P(g)^{\ddag}$, which the goal of this section. We know that
\begin{equation}
\label{closures}
P(g)^{\ddag}\,=\,\bigcup_{d\in\N}\overline{P(g)\cap\R[\x]_d},\end{equation}
and if for instance $P(g)$ has an algebraic interior then $\overline{P(g)}=P(g)^{\ddag}$.
(See e.g. Kuhlmann and Marshall \cite[Prop. 1.4]{salma} and Cimpric et al. \cite[Prop. 1.3]{cimpric}.)

Notice that  $Q(g)^{\ddag}=\bigcup_{d\in\N}\overline{Q(g)\cap\R[\x]_d}$ (\cite{salma})
and by \cite[Prop. 1.3]{cimpric}, we also have
$\overline{Q(g)}=Q(g)^{\ddag}$ if $Q(g)$ is archimedean. 

\begin{thm}
\label{th2}
Assume that $\K\subseteq\R^n$ in (\ref{setk}) has a nonempty interior and let $f\in\R[\x]$. Then:

{\rm (a)}  $f\in P(g)^{\ddag}$ if and only if there is some $d\in\N$ such that for every $\epsilon >0$, 
the polynomial 
\begin{equation}
\label{polq}
\x\mapsto f(\x)+\epsilon \left(1+\sum_{i=1}^nx_i^{2d}\right)\quad\mbox{is in $P(g)$.}
\end{equation}

{\rm (b)} The same statement as (a) also holds with $Q(g)$ instead of $P(g)$.

In other words, $q\in\R[\x]$ in (\ref{seq-closure}) for $A=P(g)$ can be taken
as $\x\mapsto (1+\sum_{i=1}^n x_i^{2d})$ independently of $\K$. The dependence of $q$ on $f$ is through the power $d$ only.
\end{thm}
\begin{proof}
(a) From (\ref{closures})  $f\in P(g)^{\ddag}$ if and only if 
$f\in \overline{P(g)\cap\R[\x]_{2d}}$ for some $d\in\N$.
Next, by Corollary \ref{cor-2}, $f\in \overline{P(g)\cap\R[\x]_{2d}}$ if and only if
the polynomial $g^\w_f\in\R[\x]_{2d}$ defined in (\ref{best-p-22}) (and which simplifies to (\ref{simplify}) when
$w_\alpha=1$ for all $\alpha\in\N^n$) is identical to $f$. But then
this implies that the polynomial $g^{P\w}_f(t)\in P_t(g)\cap\R[\x]_{2d}$ in (\ref{aux}) is such that
$\sum_{i=0}^n\lambda_i^{P\w}(t)\to0$ as $t\to\infty$ (recall that $w_\alpha=1$ for all $\alpha\in\N^n$). Let $\lambda(t):=\max_i[\lambda^{P\w}_i(t)]$ so that 
$\lambda(t)\to0$ as $t\to\infty$, and the polynomial
\[\x\mapsto f(\x)+\lambda(t)\left(1+\sum_{i=1}^n x_i^{2d}\right)\]
belongs to $P_t(g)\cap\R[\x]_{2d}$ because
\[f+\lambda(t)\left(1+\sum_{i=1}^n x_i^{2d}\right)=g^{P\w}_f(t)+\underbrace{\lambda(t)-\lambda_0^P(t)}_{\geq0}+
\sum_{i=1}^n\underbrace{(\lambda(t)-\lambda^P_f(t))}_{\geq0}x_i^{2d}.\]
Therefore, for every $\epsilon>0$,
choosing $t_\epsilon$
such that $\lambda(t_\epsilon)\leq\epsilon$ ensures that the polynomial
$f+\epsilon (1+\sum_{i=1}^nx_i^{2d})$ is in $P_{t_\epsilon}(g)\cap\R[\x]_{2d}$, which implies the desired result in (a).

The proof of (b)  is omitted as it follows similar arguments. Indeed, it was already mentioned that Theorem \ref{thmain-2} and Corollary \ref{cor-2} have obvious analogues for the quadratic module $Q(g)$.
\end{proof}
Of course in (\ref{polq}) one may replace $1+\sum_{i=1}^nx_i^{2d}$ with $1+\sum_{i=1}^n\sum_{k=1}^d\frac{x_i^{2k}}{(2k)\l}$.
In fact, and as pointed out by an anonymous referee,  in (a) one may also replace 
$1+\sum_{i=1}^nx_i^{2d}$ with any point in the interior of $P_d(g)$ in $\R[\x]_{2d}$
and in (b) with any point in the interior of $Q_d(g)$ in $\R[\x]_{2d}$. The fact that
$1+\sum_{i=1}^nx_i^{2d}$ is an interior point of $\Sigma[\x]_d$ in $\R[\x]_{2d}$
(and hence also an interior point of $P_d(g)$ and $Q_d(g)$ in $\R[\x]_{2d}$) seems to be well-known. 
For instance, it can be deduced from \cite[Theorem 3]{arch}.

\subsection{A Positivstellensatz for non compact $\K$}
As we know how to project with the $\ell_\w$-norm,
we are now able to obtain the following Positivstellensatz on $\K$.
\begin{cor}
\label{pos-stellen}
Let $\K\subseteq\R^n$ in (\ref{setk}) be nonempty interior.
Then $f\geq0$ on $\K$ if and only if for every $\epsilon>0$ there exists $d\in\N$ such that
\begin{equation}
\label{posit}
\x\mapsto f(\x)+\epsilon\left(1+\sum_{i=1}^n\sum_{k=1}^d\frac{x_i^{2k}}{(2k)\l}\right)\,\in\,P(g).\end{equation}
\end{cor}
\begin{proof}
The {\it only if part:} Recall that $P(g)=\bigcup_{d\geq0}P_d(g)$, and from Theorem \ref{thmain-4}, $\psd(\K)={\rm cl}_\w(P(g))$.
Let $g^{P\w}_f(d)\in\R[\x]_{2d}$ be the canonical $\ell_\w$-projection of $f$ onto 
$P_d(g)$ given in (\ref{best-p}), where $p_{d\w}=\lambda_0^{P\w}+\sum_{i=1}^n\sum_{k=1}^d
\lambda_{ik}^{P\w}$.
As $f\in{\rm cl}_\w(P(g))$, we necessarily have $\lim_{d\to\infty}p_{d\w}=0$, because
$\Vert f-g^{P\w}_f(d)\Vert_\w\to 0$.
Hence given $\epsilon>0$, let $d$ be such that
$\max_{i,k}\lambda_{ik}^{P\w}\leq\epsilon$. Then
\[f+\epsilon\left(1+\sum_{i=1}^n\sum_{k=1}^d\frac{x_i^{2k}}{(2k)\l}\right)
=g^{P\w}_f(d)+\underbrace{(\epsilon-\lambda^{P\w}_0)+\sum_{i=1}^n\sum_{k=1}^d
(\epsilon-\lambda^{P\w}_{ik})\frac{x_i^{2k}}{(2k)\l}}_{\in\Sigma[\x]},\]
and so $f+\epsilon(1+\sum_{i=1}^n\sum_{k=1}^d\frac{x_i^{2k}}{(2k)\l})\in P(g)$.

The {\it if part.} Let $q_d\in\R[\x]$ be the polynomial in (\ref{posit}), and let
$\x\in\K$ be fixed, arbitrary. Then $0\leq q_d(\x)\leq f(\x)+\epsilon \sum_{i=1}^n\exp{\vert x_i\vert }$.
Therefore, letting $\epsilon\to 0$ yields $f(\x)\geq0$.
\end{proof}

\section{Appendix}
\begin{lem}
\label{newcarleman}
Let $\mu$ a finite Borel measure whose sequence of moments 
$\y=(y_\alpha)$, $\alpha\in\N^n$, is such that for all $i=1,\ldots,n$, and all $k\in\N$,
$L_\y(\x_i^{2k})\leq (2k\l)M$ for some $M$.
Let $f\in\R[\x]$
be such that $L_\y(\x_i^{2t}f)\geq0$ for all $i=1,\ldots,n$, and all $t\in\N$.
Then the sequence $\z^f=(z^f_\alpha)$, $\alpha\in\N^n$, where $z^f_\alpha=L_{\z^f}(\x^\alpha):=L_\y(\x^{\alpha}f)$
for all $\alpha\in\N^n$, satisfies Carleman's condition (\ref{carleman}). 
\end{lem}
\begin{proof}
Let $1\leq i\leq n$ be fixed arbitrary, and let $2s\geq{\rm deg}f$. Observe that
whenever $\vert \alpha\vert\leq k$, 
$\vert\x\vert^\alpha\leq \vert\x_j\vert^k$ on the subset $W_j:=\{\x\in \R^n\setminus [-1,1]^n\,:\,\vert x_j\vert=\max_i\vert x_i\vert\}$. And so, $\vert f(\x)\vert\leq \Vert f\Vert_1\, \vert\x_j\vert^{2s}$ for all $\x\in W_j$. Hence,
\begin{eqnarray*}
L_{\z^f}(x_i^{2k})&=&\int f(\x)\,x_i^{2k}d\mu(\x)\\
&\leq&\int_{[-1,1]^n} \vert f(\x)\vert\,x_i^{2k}d\mu(\x)+
%\Vert f\Vert_1\,\sum_{j=1}^n\int_{\R^n\setminus [-1,1]^n} \,x_j^{2(k+s)}d\mu(\x)\\
\Vert f\Vert_1\,\sum_{j=1}^n\int_{W_j} \,x_j^{2(k+s)}d\mu(\x)\\
&\leq& M\Vert f\Vert_1+Mn\Vert f\Vert_1\,(2(k+s))\l\,\leq\,2Mn\Vert f\Vert_1\,(2(k+s))\l,
\end{eqnarray*}
and so we have
\begin{eqnarray*}
L_{\z^f}(x_i^{2k})^{-1/2k}&\geq& (2Mn\Vert f\Vert_1)^{-1/2k}\,\left(((2(k+s))\l)^{-1/2(k+s)}\right)^{(k+s)/k}\\
&\geq& \frac{1}{2}\left(((2(k+s))\l)^{-1/2(k+s)}\right)^{(k+s)/k}\\
&\geq&\frac{1}{2}\left(\frac{1}{2(k+s)}\right)^{(k+s)/k},
%\left(L_\y(x_i^{2(k+s)})^{-1/2(k+s)}\right)^{(k+s)/k},
\end{eqnarray*}
where $k\geq k_0$ is sufficiently large so that $(2Mn\Vert f\Vert_1)^{-1/2k}\geq 1/2$. 
Therefore,
\[\sum_{k=1}^\infty
L_{\z^f}(x_i^{2k})^{-1/2k}\geq \frac{1}{2}\sum_{k=k_0}^\infty\left(\frac{1}{2(k+s)}\right)^{(k+s)/k}\,=\,+\infty.\]
where the last equality follows from $(\frac{1}{2(k+s)})^{(k+s)/k}=(\frac{1}{2(k+s)})(\frac{1}{2(k+s)})^{s/k}$ and
$(\frac{1}{2(k+s)})^{s/k}\geq 1/2$ whenever $k$ is sufficiently large, say $k\geq k_1$. Hence the sequence
$\z^f$ satisfies Carleman's condition (\ref{carleman}).
\end{proof}
\subsection*{Acknowledgement}
The author wishes to thank an anonymous referee for several comments and suggestions
that helped writing the final version.

\end{document}